\documentclass[12pt]{article}

\usepackage[margin=1in]{geometry}

\usepackage{amsmath, amssymb, amsthm, amssymb}
\usepackage{amsfonts}
\usepackage{graphicx}
\usepackage[frame,ps,matrix,arrow,curve,rotate]{xy}
\usepackage{enumerate}
\usepackage{hyperref}
\usepackage{siunitx}
\usepackage{enumitem}
\usepackage{xspace}

\usepackage{blkarray}

\newtheorem{theorem}{Theorem}[section]

\newtheorem{corollary}[theorem]{Corollary}

\newcounter{claims}[theorem]
\newtheorem{claim}[claims]{Claim}

\theoremstyle{definition}
\newtheorem{definition}[theorem]{Definition}

\theoremstyle{remark}

\newcommand{\mc}[1]{\mathcal{#1}}

\newcommand{\la}{\langle}
\newcommand{\ra}{\rangle}

\DeclareMathOperator{\dom}{dom}

\newcommand{\abs}[1]{\lvert#1\rvert}
\newcommand{\innerprod}[1]{\langle#1\rangle}
\newcommand{\pset}[1]{\mathcal{P}(#1)}
\newcommand{\set}[1]{\{#1\}}

\newcommand{\sym}{\text{sym}}
\newcommand{\fix}{\text{fix}}

\renewcommand{\ge}{\geqslant}
\renewcommand{\geq}{\geqslant}

\renewcommand{\leq}{\leqslant}
\renewcommand{\emptyset}{\varnothing}

\newcommand{\liff}{\leftrightarrow}
\newcommand{\limplies}{\rightarrow}

\newcommand{\xor}{\oplus}
\newcommand{\es}{\emptyset}

\newcommand{\cge}{\succeq}
\newcommand{\cle}{\preceq}

\newcommand{\DedFinLogic}{$\mathsf{DedCardComp}$\xspace}
\newcommand{\CardCompLogic}{$\mathsf{CardComp}$\xspace}
\newcommand{\FinLogic}{$\mathsf{FinCardComp}$\xspace}

\newcommand{\ZF}{\mathsf{ZF}}
\newcommand{\ZFC}{\mathsf{ZFC}}
\newcommand{\ZFA}{\mathsf{ZFA}}
\newcommand{\AC}{\mathsf{AC}}

\begin{document}

\title{The Logic of Cardinality Comparison Without the Axiom of Choice}
\author{Matthew Harrison-Trainor\thanks{The first author was partially supported by NSF grant \mbox{DMS-2153823}.} and Dhruv Kulshreshtha\thanks{The second author was supported by the Math Chair's Discretionary Funds at the University of Michigan Mathematics REU Program.}}

\maketitle

\begin{abstract}
	We work in the setting of Zermelo-Fraenkel set theory without assuming the Axiom of Choice. We consider sets with the Boolean operations together with the additional structure of comparing cardinality (in the Cantorian sense of injections). What principles does one need to add to the laws of Boolean algebra to reason not only about intersection, union, and complementation of sets, but also about the relative size of sets? We give a complete axiomatization.
	
	A particularly interesting case is when one restricts to the Dedekind-finite sets. In this case, one needs exactly the same principles as for reasoning about imprecise probability comparisons, the central principle being Generalized Finite Cancellation (which includes, as a special case, division-by-$m$). In the general case, the central principle is a restricted version of Generalized Finite Cancellation within Archimedean classes which we call Covered Generalized Finite Cancellation.
\end{abstract}

\section{Introduction}

Assume throughout that $\ZF$ is consistent. Under the Axiom of Choice ($\AC$), cardinal arithmetic is highly structured. Every set is in bijection with a cardinal, and these are totally ordered. Thus for any two sets $A$ and $B$, either $|A| \geq |B|$ or $|B| \geq |A|$. Without the Axiom of Choice, this is no longer the case. Indeed, Jech \cite{Jech66} and Takahashi \cite{Takahashi67} showed that the situation is as far from this as possible.

 \begin{theorem}[Jech \cite{Jech66} and Takahashi \cite{Takahashi67}; see Theorem 11.1 of \cite{Jech73}]\label{thm:Jech}
	Let $(\mc{P},\preceq)$ be a partial preorder. Then there is a model $\mc{U}$ of $\ZF$, and sets $(A_p)_{p \in \mc{P}}$ in $\mc{U}$, such that
	\[ p \preceq q \Longleftrightarrow |A_p| \leq |A_q|.\]
\end{theorem}

\noindent One can think of this result as giving a complete order-theoretic characterization of cardinality comparison in $\ZF$. In this paper, we will expand this to also consider the Boolean operations of intersection, union, and complement, that is, to completely characterize additive cardinal arithmetic in $\ZF$. From this perspective, our main result is as follows.

\begin{theorem}
	Let $B$ be a finite Boolean algebra with $\top$ as the top element and $\bot$ the bottom element, and let $\cge$ be a binary relation on $B$. The following are equivalent:
	\begin{enumerate}
		\item There is a model $\mc{U}$ of $\ZF$ set theory, a non-empty set $X$ in $\mc{U}$, and a field of sets $\mc{F} = \{A_p \subseteq X : p \in B\}$ over $X$ representing $B$ and such that for $p,q \in B$,
		\[ p \preceq q \Longleftrightarrow |A_p| \leq |A_q|.\]
        Here, a field of sets on $X$ is a non-empty collection of subsets of $X$ closed under finite unions, intersections, and relative complements.
		\item $(B,\preceq)$ satisfies the following conditions:
		\begin{itemize}
			\item not $\bot \cge \top$;
			\item for all $b \in B$, $b \cge \bot$;
			\item $\cge$ is reflexive;
			\item $\cge$ is transitive;\footnote{In fact, transitivity follows as a special case of the next condition.}
			\item for any two sequences of elements $a_1, a_2, \dots , a_k,\underbrace{e,\ldots,e}_{l}$ and $b_1, b_2, \dots, b_k,\underbrace{f,\ldots,f}_{l}$ from $B$ of equal length, if
			\begin{enumerate}
			    \item every atom of $B$ is below (in the
			order of the Boolean algebra) exactly as many $a$'s and $e$'s as $b$'s and $f$'s (counting multiplicity among the $e$'s and $f$'s),
			    \item $a_i
			\cge b_i$ for all $i \in\{1, \dots , k\}$, and
			    \item for each $i$, $a_i$ is contained in the smallest ideal containing $f$ and closed downwards under $\cle$.
			\end{enumerate}
			then $f \cge e$.
		\end{itemize}
	\end{enumerate}
\end{theorem}

\noindent The last condition here is a variation on the \textit{Generalized Finite Cancellation} principle which comes from the study of imprecise probability. We will explain this condition and its connections in more detail later.

\medskip

Another way of viewing this work, from a logical perspective, is as providing a complete axiomatization for reasoning about the relative sizes of sets in $\ZF$ in a formal set-theoretic language. Just as the laws for reasoning about intersection, union, and complementation of sets are captured by the laws of Boolean algebra, what are the laws one must add to the laws of Boolean algebra to capture reasoning about the relative sizes of sets?

Ding, Harrison-Trainor, and Holliday \cite{DHTH20} determined the principles required for reasoning about relative cardinality in the presence of the Axiom of Choice. There, the key difficulty was to reconcile the principles for reasoning about finite sets with the principles for reasoning about infinite sets. In essence, Cantorian reasoning about the relative sizes of \textit{finite} size is the same as \textit{probabilistic} reasoning about the relative likelihoods of events, while Cantorian reasoning about the relative sizes of \textit{infinite} sets is the same as what is called \textit{possibilistic} reasoning \cite{Dubois1988} about the relative likelihoods of events. Each type of likelihood reasoning had previously been axiomatized by itself \cite{Segerberg1971, Gardenfors1975, Burgess2010, Cerro1991}; the chief difficulty in \cite{DHTH20} was to combine them.

In this paper, we determine the principles required for reasoning about relative cardinality without assuming the Axiom of Choice. The difficulties here are of a quite different nature. Rather than reconciling the finite and the infinite, the difficulty now lies in finding a complete set of principles for reasoning about the infinite. With the Axiom of Choice, cardinal arithmetic with infinite sets is almost trivial: for example, sets are totally ordered by cardinality and $|X \cup Y| = \max(|X|,|Y|)$. On the other hand, without the Axiom of Choice, there are sets of incomparable cardinality, and it is possible to have $|X \cup Y| > |X|,|Y|$. Nevertheless, cardinal arithmetic is not totally wild and there is a still some structure, such as division-by-$m$: If $|m \times X| \leq |m \times Y|$, then $|X| \leq |Y|$, for any $m \in \mathbb{N}_{>0}$. Thus the problem is to find a set of principles, including division-by-$m$, that completely characterise additive cardinal arithmetic.

Formally, we work with a language (see Definition \ref{def:language}) that allows us to build terms using the standard set-theoretic operations of union, intersection, and complement, and to express that a set $s$ is at least as big as a set $t$: $|s|\geq |t|$. Thus, we work with a comparative notion of size, prior to the reification of sizes as cardinal numbers. The semantics is given by the Cantorian definition: $|s|\geq |t|$ is true if and only if there is an injection from $t$ into $s$. Then we prove soundness and completeness (Theorems \ref{thm:soundness}, \ref{thm:soundness-inf}, \ref{thm:completeness}, and \ref{thm:completeness-inf}) results for our logic (Definitions \ref{dedfinlogic} and \ref{cardcomplogic}) in this language.

\medskip

Many of the axioms and rules for this logic are exactly what one would expect, e.g., that union, intersection, and complements obey the laws of Boolean algebra, cardinality comparison is transitive, etc. There is one key axiom, which we will now describe, which plays the most important role. Under the Axiom of Choice, one divides the universe into finite sets and infinite sets. Without the Axiom of Choice, there are two different definitions of finite which are no longer equivalent. First, we say that a set is \textit{finite} if it is in bijection with some natural number. Second, we say that a set is \textit{Dedekind-finite} if it is strictly larger in size than any of its proper subsets. Every finite set is Dedekind-finite, but not necessarily vice versa.

Suppose first that we restrict our attention to the simpler case of only Dedekind-finite sets. What are the principles for reasoning about relative sizes of Dedekind-finite sets? Recall that reasoning about relative sizes of \textit{finite} sets is the same as \textit{probabilistic} reasoning about the relative likelihoods of events; this is in a sense where one is absolutely certain, given two events, of which is more likely than the other. We show that reasoning about relative sizes of \textit{Dedekind-finite} sets is the same as \textit{imprecise probabilistic} reasoning about the relative likelihood of events, i.e., reasoning in a setting where there can be uncertainty about the relative likelihood of two events. (For the reader unfamiliar with imprecise probability, we provide more detail in Section \ref{sec:prop-measures-models}.) The central principle, Generalized Finite Cancellation, was isolated in the context of imprecise probability by Rios Insua \cite{Insua92} and Alon and Lehrer \cite{AL14}. We show that the same principle is true of cardinality comparison of Dedekind-finite sets. In terms of cardinality, it says the following.

\begin{definition}\label{def:balanced} We say that two sequences of sets $\innerprod{E_1,\ldots,E_k}$ and $\innerprod{F_1,\ldots,F_k}$ are \textit{balanced}, and write $\innerprod{E_1,\ldots,E_k} =_0 \innerprod{F_1,\ldots,F_k}$, if and only if for all $s$, the cardinality of $\set{i \ | \ s\in E_i}$ is equal to the cardinality of $\set{i \ | \ s\in F_i}$; that is, if every $s$ appears the same number of times on the left side as on the right side.
\end{definition}

\begin{theorem}[Generalized Finite Cancellation]\label{thm:GFC}
	Dedekind-finite sets satisfy Generalized Finite Cancellation: Suppose that
	\[\innerprod{A_1,\ldots,A_k,\underbrace{E,\ldots,E}_{\ell}} =_0 \innerprod{B_1,\ldots,B_k,\underbrace{F,\ldots,F}_{\ell}},\]
	that $B_1,\ldots,B_k$ are Dedekind-finite, and that $\abs{A_i} \geq \abs{B_i}$ for each $i$. Then $\abs{E}\leq\abs{F}$.
\end{theorem}

\noindent Thus reasoning about the cardinality of Dedekind-finite sets turns out to be exactly the same as reasoning about imprecise probability. (We note that Generalized Finite Cancellation is, as the name suggests, a generalization of the principle of Finite Cancellation. Finite Cancellation is just the special case $\ell = 1$, and it is, together with Totality, the central principle for reasoning about the cardinalities of finite sets; see Section \ref{sec:finite}.)

Now in the general case, where we consider all sets rather than just Dedekind-finite sets, the key principle is a certain restriction of Generalized Finite Cancellation. This principle has not appeared in the literature on imprecise probability.

\begin{theorem}[Covered Generalized Finite Cancellation]\label{thm:CGFC}
	All sets satisfy Covered Generalized Finite Cancellation: Suppose that 
	\[\innerprod{A_1,\ldots,A_k,\underbrace{E,\ldots,E}_{\ell}} =_0 \innerprod{B_1,\ldots,B_k,\underbrace{F,\ldots,F}_{\ell}},\]
	that each $|B_i| \leq |A_i|$, and that for each $i$, there is $n$ such that $|A_i| \leq |n \times F|$. Then $|E| \leq |F|$.
\end{theorem}

The structure of this paper is as follows. In Section \ref{sec:no-choice} we describe cardinal arithmetic in the absence of the Axiom of Choice and prove Theorems \ref{thm:GFC} and \ref{thm:CGFC} as pure set-theoretic results. Then in Section \ref{sec:formal} we give the formal definitions of our language and models, explicitly give our logics \DedFinLogic (Definition \ref{dedfinlogic}) and \CardCompLogic (Definition \ref{cardcomplogic}) for reasoning about additive cardinality comparison for Dedekind-finite and arbitrary sets respectively. We also prove in Theorems \ref{thm:soundness} and \ref{thm:soundness-inf} that these logics are sound. The rest of the paper is concerned with proving completeness for these two logics:

\begin{theorem}\label{thm:completeness}
    \DedFinLogic is complete with respect to Dedekind-finite  cardinality models.
\end{theorem}

\begin{theorem}\label{thm:completeness-inf}
    \CardCompLogic is complete with respect to cardinality models.
\end{theorem}

\noindent We note that these are weak completeness results, i.e., they show that any finite consistent set of formulas has a model. It is an open problem, even in the case of $\ZFC$, to find a logic which is strongly complete.

The proofs of these are divided into several steps. First, we use completeness results for models of imprecise probabilistic reasoning. For Dedekind-finite sets, this is straight from the literature on imprecise probability, while for arbitrary sets we use analogues that allow a set to be given measure $\infty$ and prove a new completeness theorem. These are described in Section \ref{sec:prop-measures-models}. In Section \ref{sec:prob-to-set} we show how to build, from these probabilistic models, a model of set theory with urelements. Finally, in Section \ref{sec:urelements} (which comes before Section \ref{sec:prob-to-set} because it contains the background on set theory with urelements), we use the Jech-Sochor Embedding Theorem to convert a model of set theory with urelements into an equivalent model without urelements.\footnote{By \textit{equivalent}, we mean $\equiv_{\mc{L}}$, as defined in Section \ref{sec:language}.} Thus we convert the completeness results for the probabilistic models into completeness results for our set-theoretic models.

\medskip

One of the key features of the axiomatization in \cite{DHTH20} under the Axiom of Choice was that there are certain facts about a set $x$ which imply that it must be finite (such as if there is a set $y$ with $|y| < |x|$ but $|x| < |x \cup y|$), and other facts about a set $x$ that imply that it must be infinite (such as if there is a set $y$ with $y \nsubseteq x$ such that $|x \cup y| = |x|$). Thus there are valid principles for reasoning about infinite sets which are invalid for finite sets, and vice versa; this is essentially an expression of the fact that cardinal arithmetic for finite and infinite sets are so different as to be orthogonal.

On the other hand, it is a consequence of our axiomatizations that without the axiom of choice any principle for reasoning about Dedekind-infinite sets is also valid for reasoning about Dedekind-finite infinite sets, and any principle for reasoning about Dedekind-finite infinite sets is also valid for reasoning about finite sets. (The converses are however false.) Thus the principles for reasoning about finite, Dedekind-finite infinite, and Dedekind-infinite sets are strictly nested rather than orthogonal.

\section{Cardinal arithmetic without the axiom of choice}\label{sec:no-choice}

We begin with a brief overview of cardinal arithmetic without the Axiom of Choice. Recall that without the axiom of choice, the ordering by cardinality is no longer total; in fact, the Axiom of Choice is equivalent to saying that for all sets $A$ and $B$, $|A|\leq|B|$ or $|B| \leq |A|$. This ordering is still partial, as seen by the following theorem (that is provable in $\ZF$), in addition to basic facts about injections.

\begin{theorem}[Cantor-Schr\"oder-Bernstein Theorem]\label{thm:CSB}
	If $\abs{X} \leq \abs{Y}$ and $\abs{Y} \leq \abs{X}$, then $\abs{X}=\abs{Y}$.
\end{theorem}

\noindent On the other hand, it is no longer true  that if there is a surjection from one set to another, then there is an injection the other way \cite{Tarski65}; that is, $|X| \leq |Y|$ is not the same as saying that there is a surjection from $Y$ onto $X$. Thus it is not true that if there is a surjection from $A$ onto $B$ and a surjection from $B$ onto $A$, then $A$ and $B$ are in bijection.

Recall that, with the Axiom of Choice, $|A \cup B| = \max\set{|A|,|B|}$ whenever at least one of $A$ and $B$ is infinite, so the additive structure on infinite sets is trivial. Without the Axiom of Choice this is no longer the case, and further interesting structure is revealed. A key example of this is division by $m$:
\begin{theorem}[Division by $m$, Lindenbaum (unpublished), Tarski \cite{Tarski49}]\label{thm:div-by-m}
	If $\abs{m\times A}\leq\abs{m\times B}$, then $\abs{A}\leq\abs{B}$.
\end{theorem}
\noindent This is true for trivial reasons under the Axiom of Choice, but it becomes a non-trivial theorem without the Axiom of Choice. The history of division-by-$m$ is long. It was first proved for equalities by Lindenbaum, but the proof was unpublished and lost. A proof for the $m=2$ case was given by Bernstein \cite{Bernstein05} and Sierpinski \cite{Sierpinski22}. Later, Sierpinski also gave a proof for dividing an inequality by two \cite{Sierpinski47}. A proof for the $m=3$ case was given by Tarksi \cite{Tarski49}, which can be generalized for any $m\neq0$. Finally, an easy-to-read proof for dividing an inequality by three given by Doyle and Conway can be found in \cite{DC94}.



\subsection{Dedekind-finite sets and generalized finite cancellation}

There are two possible definitions of finite which were, for much of the history of mathematics, assumed to be equivalent. First, we say that a set is \textit{finite} if it is in bijection with some natural number $n = \set{0,1,\ldots,n-1} \in \omega$; otherwise, it is \textit{infinite}. Dedekind introduced an alternative definition in \cite{Dedekind1888}. We say that a set is \textit{Dedekind-finite} if it is not in bijection with a proper subset of itself, or equivalently, if it does not admit an injection of $\omega$; otherwise, it is \textit{Dedekind-infinite}.

One can use the Axiom of Choice to prove that the two definitions coincide: a set is finite if and only if it is Dedekind-finite. In general (i.e., without assuming the Axiom of Choice) a finite set is always Dedekind-finite, but a Dedekind-finite set might not be finite. That is, there may exist an infinite but Dedekind-finite set, sometimes referred to as a \textit{Dedekind set}. Any finite set is of lower cardinality than any infinite set, including a Dedekind-finite one.

One particular type of Dedekind-finite set is an amorphous set. An infinite set $A$ is said to be \textit{amorphous} if there do not exist infinite disjoint sets $B,C \subseteq A$ such that $A = B \sqcup C$. Every amorphous set is Dedekind-finite, but not necessarily vice versa. (For example, the union of two infinite amorphous sets is not amorphous, while the union of two infinite Dedekind-finite sets is still Dedekind-finite.)

Dedekind-finite sets behave like finite sets in some ways, but not others. For example, there are Dedekind-finite sets which are of incomparable cardinality (and indeed, in Theorem \ref{thm:Jech}, if the partial order $\mc{P}$ is finite the sets $A_p$ may be taken to be Dedekind-finite) whereas the cardinalities of finite sets are all natural numbers and hence comparable. On the other hand, Dedekind-finite sets are like finite sets in that they satisfy a subtraction principle.

\begin{theorem}[Subtraction, see \cite{DC94}]\label{thm:subtraction}
	For any Dedekind-finite set $Z$, if $\abs{X\sqcup Z} = \abs{Y\sqcup Z}$, then $\abs{X} = \abs{Y}$.    
\end{theorem}

Finally, we will prove that Dedekind-finite sets satisfy the principle of \textit{Generalized Finite Cancellation (GCF)}. GFC first appeared in work of Rios Insua \cite{Insua92} and Alon and Lehrer \cite{AL14} on characterisations of relations of imprecise probability. We prove the same principle, though in a completely different context: We have Dedekind-finite sets instead of events in a probability space, and we compare using cardinality instead of relative likelihood. One can also think of Generalized Finite Cancellation as being a common generalization of both Division-by-$m$ (Theorem \ref{thm:div-by-m}) and Subtraction (Theorem \ref{thm:subtraction}).

Recall from Definition \ref{def:balanced} that we say that two sequences of sets $\innerprod{E_1,\ldots,E_k}$ and $\innerprod{F_1,\ldots,F_k}$ are \textit{balanced}, and write $\innerprod{E_1,\ldots,E_k} =_0 \innerprod{F_1,\ldots,F_k}$, if and only if for all $s$, the cardinality of $\set{i \ | \ s\in E_i}$ is equal to the cardinality of $\set{i \ | \ s\in F_i}$; that is, if every $s$ appears the same number of times on the left side as on the right side.


{
\renewcommand{\thetheorem}{\ref{thm:GFC}}
\begin{theorem}[Generalized Finite Cancellation]
  Suppose that
	\[\innerprod{A_1,\ldots,A_k,\underbrace{E,\ldots,E}_{\ell}} =_0 \innerprod{B_1,\ldots,B_k,\underbrace{F,\ldots,F}_{\ell}},\]
	that $B_1,\ldots,B_k$ are Dedekind-finite, and that $\abs{A_i} \geq \abs{B_i}$ for each $i$. Then $\abs{E}\leq\abs{F}$.
\end{theorem}
\addtocounter{theorem}{-1}
}
\begin{proof}
	Among the elements of $\innerprod{A_1,\ldots,A_k,{E,\ldots,E}}$ there may be repeated elements. We may replace the $i$th appearance of each element $a$ in the sequence by the ordered pair $(a,i)$. In doing so, we obtain a new sequence of pairwise disjoint sets $\innerprod{A_1',\ldots,A_k',E_1,\ldots,E_\ell}$, and each set in this sequence is in bijection with the corresponding set of the original sequence. Similarly, we can replace the second sequence by $\innerprod{B_1',\ldots,B_k',F_1,\ldots, F_\ell}$. We still have that $\abs{A_i'} \geq \abs{B_i'}$ for each $i$, and that the sequences are balanced:
	\[\innerprod{A_1',\ldots,A_k',E_1,\ldots,E_\ell} =_0 \innerprod{B_1',\ldots,B_k',F_1,\ldots, F_\ell}. \]
	Indeed we can now replace $A_1',\ldots,A_k'$ and $B_1',\ldots,B_k'$ by their disjoint unions $A = A_1' \cup \cdots \cup A_k'$ and $B = B_1' \cup \cdots \cup B_k'$. Then $\abs{A} \geq \abs{B}$, and we have balanced sequences
	\[\innerprod{A,E_1,\ldots,E_\ell} =_0 \innerprod{B,F_1,\ldots, F_\ell}. \]
	Moreover, $B$ is Dedekind-finite. We will show that $\abs{\ell \times E} = \abs{E_1 \cup \cdots \cup E_\ell} \leq \abs{F_1 \cup \cdots \cup F_\ell} = \abs{\ell \times F}$, from which it follows by dividing by $\ell$ (Theorem \ref{thm:div-by-m}) that $\abs{E} \leq \abs{F}$. So, without loss of generality, we may assume (replacing $E_1 \cup \cdots \cup E_\ell$ by $E$ and $F_1 \cup \cdots \cup F_\ell$ by $F$) that we are in the case
	\[\innerprod{A,E} =_0 \innerprod{B,F}. \]
	Now from here one could use subtraction, but we will give a full proof both for completeness and because it will soon be generalized.
	
	Fix an injection $f : B \to A$ witnessing that $\abs{A} \geq \abs{B}$. We will define an injection $g: E \to F$. Given $x \in E$, we must define $g(x)$. By the balancing assumption, $x \in B \cup F$. If $x \in F$, let $g(x) = x$. Otherwise, if $x \in B$, $f(x) \in A$. By the balancing assumption, $f(x) \in B \cup F$; if $f(x) \in F$, set $g(x) = f(x)$. Otherwise, $f(x) \in B$ and so $f(f(x)) \in A$. Continue until, for some $k$, $f^k(x) \in F$ and we define $g(x) = f^k(x)$. If there does not exist $k$ such that $f^k(x) \in F$, then $x,f(x),f(f(x)),\ldots$ are all in $B$. Since $f$ is injective, this sequence has no repetition, and so $B$ contains an $\omega$-sequence. This contradicts the fact that $B$ is Dedekind finite. So the construction eventually terminates and defines $g(x)$.
	
	We must argue that $g$ is injective. For distinct $x$ and $y$ in $E$, consider the two sequences $x,f(x),f(f(x)),\ldots,f^k(x) = g(x)$ and $y,f(y),f(f(y)),\ldots,f^\ell(y) = g(y)$. Each element of these two sequences, except the first elements $x$ and $y$, are in $A$. Thus $x$ does not appear anywhere in the sequence for $y$, and $y$ does not appear anywhere in the sequence for $x$. Since $f$ is injective, it must be that $g(x) \neq g(y)$. Hence $g$ is injective. Thus we have shown that $\abs{E}\leq \abs{F}$ as desired. 
\end{proof}

\subsection{Dedekind-infinite sets and covered generalized finite cancellation}

It is easy to see that infinite sets do not have to satisfy GFC; indeed
\[ \la \mathbb{N},\varnothing \ra =_0 \la 2 \mathbb{N},\mathbb{N} \setminus 2\mathbb{N} \ra\]
and $|\mathbb{N}| \leq |2 \mathbb{N}|$, but $|\mathbb{N}\setminus 2\mathbb{N}| > \varnothing$. 

Instead we will prove that arbitrary sets satisfy GFC under an additional condition that none of the sets involved are too much larger than the set $F$ for which we want to conclude that $|E| \leq |F|$. We call this principle \textit{Covered Generalized Finite Cancellation (CGFC)} as the set $F$ in a sense \textit{covers} the other sets.

{
\renewcommand{\thetheorem}{\ref{thm:CGFC}}
\begin{theorem}[Covered Generalized Finite Cancellation]
	Suppose that 
	\[\innerprod{A_1,\ldots,A_k,\underbrace{E,\ldots,E}_{\ell}} =_0 \innerprod{B_1,\ldots,B_k,\underbrace{F,\ldots,F}_{\ell}},\]
	that each $|B_i| \leq |A_i|$, and that for each $i$, there is $n$ such that $|A_i| \leq |n \times F|$. Then $|E| \leq |F|$.
\end{theorem}
\addtocounter{theorem}{-1}
}
\begin{proof}
	First, we can argue as in Theorem \ref{thm:GFC} (by replacing repeated elements by new copies, and dividing by $\ell$ as in Theorem \ref{thm:div-by-m}) that we may assume that we are in the following simpler case: There is a balanced sequence 
	\[ \la A,E \ra =_0 \la B,F \ra \]
	with $|B| \leq |A| \leq |n \times F|$, for which we want to show that $|E| \leq |F|$.
	
	
	Fix an injection $f : B \to A$ witnessing that $\abs{A} \geq \abs{B}$, and $g : A \to n \times F$ witnessing that $\abs{A} \leq \abs{n \times F}$. First we will split $E$ up into a disjoint union $E = X \cup Y$ such that there is an injection $h_X : X \to F$ and an injection $h_Y : Y \times \omega \to n \times F$. Following this, we will combine them into a single injection $h : E = X \cup Y \to F$. Intuitively, the fact that $Y$ not only injects into $F$, but injects many times, will give us sufficient room.
	
	Given $x \in E$, by the balancing assumption, $x \in B \cup F$. If $x \in F$, put $x \in X$ and set $h_X(x) = x$. Otherwise, if $x \in B$, $f(x) \in A$. By the balancing assumption, $f(x) \in B \cup F$; if $f(x) \in F$, put $x \in X$ and set $h_X(x) = f(x)$. Otherwise, $f(x) \in B$ and so $f(f(x)) \in A$. Continue for as long as possible (even forever) until possibly, for some $k$, $f^k(x) \in F$ and we put $x \in X$ and define $h_X(x) = f^k(x)$. If there is no $k$ such that $f^k(x) \in F$, then put $x \in Y$. In that case, it must be that $x,f(x),f(f(x)),\ldots$ are all in $B$. Recalling that we have an injection $g \colon A \to n \times F$, define $h_Y: Y \times \omega \to n \times F$ by $h_Y(x,m) = g(f^m(x))$. (This is similar to what we did in Theorem \ref{thm:GFC}, except in that case the assumption of Dedekind-finiteness meant that the process always stopped, so that in the notation of this proof we would have that the set $X$ is all of $E$ and $Y$ is empty.) 
	
	We must argue that $h_X$ and $h_Y$ are injective. For each $x \in E$, we constructed above a sequence $x,f(x),f(f(x)),\ldots$ where either the sequence goes on forever with all elements in $B$ (if $x \in Y$) or all of the elements are in $B$ except for the last element which is $h_X(x)$ and is in $F$. First, we argue that for distinct $x$ and $y$ in $E$, the two sequences $x,f(x),f(f(x)),\ldots$ and $y,f(y),f(f(y)),\ldots$ associated to $x$ and $y$ are disjoint. Each element of these two sequences, except the first elements $x$ and $y$, are in $A$. Thus $x$ does not appear anywhere in the sequence for $y$, and $y$ does not appear anywhere in the sequence for $f$. Since $f$ is injective, the two sequences must be disjoint and have no repetition. In particular, if $x,y \in X$, then $h_X(x)$ and $h_Y(y)$ are the terminal elements of the sequences associated to $x$ and $y$ respectively, and hence are distinct. If $x,y \in Y$, and $k,\ell \in \omega$, then $f^k(x)$ and $f^\ell(y)$ are distinct unless $x = y$ and $k = \ell$; so since $g$ is injective, $h_Y(x,k) = g(f^k(x))$ and $h_Y(y,\ell) = g(f^\ell(y))$ are distinct.
	
	Now we will explain how to combine $h_X$ and $h_Y$ into a single injection $h : E = X \cup Y \to F$. Let $F' \subseteq F$ be that part of $F$ that appears in the image of $h_Y$, so that $h_Y$ restricts to a bijection $h_Y : Y \times \omega \to n \times F'$. Thus $|F'| \leq |Y \times \omega|$, and so
	\[ \abs{Y \sqcup F'} \leq \abs{Y \times \omega}.\]
	But then, since $|\omega \times \omega| = |\omega|$,
	\[\abs{(Y \sqcup F') \times \omega} \leq \abs{Y \times \omega} \leq n \times F'.\]
	It follows that $|(Y \sqcup F') \times n| \leq |n \times F'|$ and so, by division by $n$, that $|F'| \leq |Y \sqcup F'| \leq |F'|$. Thus, by Theorem \ref{thm:CSB}, $|Y \sqcup F'| = |F'|$. Now let $X_1 = h_X^{-1}(F')$ and $X_2 = X - X_1$. Then $|X_1| = |F'|$, and so $|Y \sqcup X_1| = |F'|$. Also, $h_X(X_2) \subseteq F - F'$, and so $|X_2| \leq |F-F'|$. Thus $|E| = |X \cup Y| = |X_1 \sqcup X_2 \sqcup Y| \leq |F|$.
\end{proof}

\section{Formal language, models, and axioms}\label{sec:formal}

\subsection{Language}\label{sec:language}

We will be defining a language that can be used to talk about sets and their Boolean combinations, and to compare their cardinalities.

\begin{definition}\label{def:language}
	Given a set $\Phi$ of \textit{set labels}, the \textit{set terms} $t$ and \textit{formulas} $\varphi$ of the language $\mathcal{L}$ are generated by the following grammar:
	\begin{align*}
	t &:: = a \ | \ t^c \ | \ (t \cap t) \\
	\varphi &:: = \abs{t}\geq\abs{t} \ | \ \neg\varphi \ | \ (\varphi \land \varphi)
	\end{align*}
	where $a \in \Phi$. The other sentential connectives $\lor, \limplies,$ and $\liff$ are defined as usual, and we use $\varphi \xor \psi$ as an abbreviation for $(\varphi \lor \psi) \land \neg(\varphi \land \psi)$. Standard set-theoretic notion may be defined as follows:
	\begin{itemize}
		\item $\es := t\cap t^c$;
		\item $t \subseteq s := \abs{\es} \geq \abs{t \cap s^c}$;
		\item $t = s := (t \subseteq s \land s \subseteq t)$ and $t \neq s := \neg(t=s)$;
		\item $t \nsubseteq s := \neg(t\subseteq s)$ and $t\subsetneq s := (t \subseteq s \land s \nsubseteq t)$.
	\end{itemize}
	We also use $\abs{s} \leq \abs{t}$ for $\abs{t}\geq \abs{s}$, $\abs{s}>\abs{t}$ for $\neg(\abs{t}\geq\abs{s})$, and $\abs{s}=\abs{t}$ for $\abs{s}\geq\abs{t}\land \abs{t}\geq\abs{s}$.
\end{definition}

Throughout this paper (and starting in the next subsection) we will define a number of different kinds of models, and what it means for an $\mc{L}$-formula to be true in those models. Given any two models $\mc{M}$ and $\mc{N}$ (possibly of different kinds), we write $\mc{M} \equiv_{\mc{L}} \mc{N}$ if they satisfy the same $\mc{L}$-formulas.


\subsection{Cardinality Models}

Recall from the introduction the definition of a field of sets:

\begin{definition}
	A \textit{field of sets} is a pair $\innerprod{X,\mathcal{F}}$ where $X$ is a non-empty set and $\mathcal{F}$ is a non-empty collection of subsets of $X$ closed under union, intersection, and set-theoretic complementation (relative to $X$).
\end{definition}

We will now be defining our first type of model for the language $\mc{L}$: the (pure) cardinality model. This model will involve a field of sets $\innerprod{X, \mc{F}}$, along with this ``naming" function $V$ which assigns to each set label, an actual set in $\mc{F}$. More formally:

\begin{definition}	
	A \textit{(pure) cardinality model} is a quadruple $\mathcal{N} = \innerprod{\mathcal{W},X,\mathcal{F},V}$, where $\mathcal{W}$ is a model of $\ZF$, $X$ is a non-empty set in $\mathcal{W}$, $\innerprod{X, \mathcal{F}}$ is a field of sets in $\mc{W}$, and $V : \Phi \to \mathcal{F}$.
\end{definition}

We say pure cardinality model to distinguish these models from the urelement cardinality models we introduce later, where $\mc{W}$ is allowed to be a model of $\ZF$ with urelements, though we may drop the descriptor ``pure'' in contexts which are free from any urelements.

We say a cardinality model is a \textit{finite/Dedekind-finite cardinality model} if all of the sets in its field of sets are finite/Dedekind-finite.

\begin{definition}\label{defn-of-model}
	Given a cardinality model $\mathcal{N} = \innerprod{\mathcal{W},X,\mathcal{F},V}$, we define a function $\hat{V}$, which assigns to each set term a set in $\mathcal{F}$, by:
	\begin{itemize}
		\item $\hat{V}(a) = V(a)$ for $a\in \Phi$
		\item $\hat{V}(t^c) = X - \hat{V}(t)$
		\item $\hat{V}(t\cap s) = \hat{V}(t)\cap\hat{V}(s)$
	\end{itemize}
	We then define a satisfaction relation $\models$ as follows:
	\begin{itemize}
		\item $\mathcal{N} \models \abs{t} \geq \abs{s}$ if and only if $\mathcal{W} \models \abs{\hat{V}(t)}\geq \abs{\hat{V}(s)}$;
		\item $\mathcal{N} \models \neg\varphi$ if and only if $\mathcal{N} \not\models \varphi$;
		\item $\mathcal{N} \models \varphi \land \psi$ if and only if $\mathcal{N} \models \varphi$ and $\mathcal{N} \models \psi$.
	\end{itemize}
	Given a class $K$ of models, $\varphi$ is \textit{valid} over $K$ if and only if $\mathcal{N}\models\varphi$ for all $\mathcal{N}\in K$.
\end{definition}

\subsection{Logic for Finite Sets}\label{sec:finite}

We include, for completeness, an axiomatization for the logic of cardinality comparison for finite sets. For finite sets, one does not have to worry about constructing new models of set theory, or whether or not the axiom of choice is true. The key principle is the Finite Cancellation principle of Scott \cite{Scott64}:
\begin{description}
    \item[Finite Cancellation:] Suppose that
	\[\innerprod{A_1,\ldots,A_k,E} =_0 \innerprod{B_1,\ldots,B_k,F},\]
is a balanced sequence of finite sets, and that $\abs{A_i} \geq \abs{B_i}$ for each $i$. Then $\abs{E}\leq\abs{F}$.
\end{description}

\begin{definition}\label{finlogic} The logic for finite sets \FinLogic is the logic for $\mathcal{L}$ with the following axiom schemas:
	\begin{enumerate}[label=(D\arabic*)]
		\item All substitution instances of classical propositional tautologies;\label{axiomf:substitutions}
		\item $\neg \abs{\es}\geq\abs{\es^c}$ (Non-triviality);\label{axiomf:non-triviality}
		\item $\abs{s}\geq\abs{\es}$ (Positivity);\label{axiomf:positivity}
		\item $\abs{s}\geq\abs{s}$ (Reflexivity);\label{axiomf:reflexivity}
		\item $\abs{s}\geq\abs{t} \land \abs{t}\geq\abs{u} \limplies \abs{s}\geq\abs{u}$ (Transitivity);\label{axiomf:transitivity}
		\item $\abs{s} \geq \abs{t}$ or $\abs{t} \geq \abs{s}$ (Totality);
		\item[(FC)] FC$_{n}(s_1,\ldots,s_n,e; t_1,\ldots,t_n,f)$;
		\makeatletter\protected@edef\@currentlabel{(FC)}\makeatother\label{axiomf:FC} 
	\end{enumerate}
	and the following rules
	\begin{enumerate}
		\item[(R1)] If $\varphi$ and $\varphi \to \psi$ are theorems, then so is $\psi$. 
        (Modus Ponens);
        \label{axiomf:modus-ponens}
		\item[(R2)] If $t=0$ is provable in the equational theory of Boolean algebras, then $\abs{\es} \geq \abs{t}$ is a theorem. \label{axiomf:emptyset}
	\end{enumerate}
\end{definition}

FC$_n(s_1,\ldots,s_n,e; t_1,\ldots,t_n,f)$ is the formal expression of the Finite Cancellation principle for the sequences $\la s_1,\ldots,s_n,e\ra$ and $\la t_1,\ldots,t_n,f \ra$. Because we cannot talk about elements of sets in our language, we must be somewhat clever to express FC. The standard way to do this is as follows. First, for each $j$ such that $1 \leq j \leq n+1$, define the term $\mc{S}_j$ as the union of the terms of the form $s_1^{c_{s,1}} \cap \ldots \cap s_n^{c_{s,k}} \cap e^{c_{e}}$, where exactly $j$ many $c_*$'s are $c$ (for set-theoretic complement) and the rest are empty. So, intuitively $\mc{S}_j$ denotes the set of elements which are in exactly $j$ many sets among $s_1,\ldots,s_k,e$. Similarly, define $\mc{T}_j$ by replacing the $s$'s with $t$'s and the $e$'s with $f$'s. Then, FC$_{n}(s_1,\ldots,s_n,e; t_1,\ldots,t_n,f)$ is defined by:
\[ \Big( \bigwedge\limits_{i=1}^{n+1} \mc{S}_i = \mc{T}_i \Big) \limplies \Big(\Big(\bigwedge\limits_{i=1}^n \abs{s_i} \geq \abs{t_i} \Big) \limplies \abs{e} \leq \abs{f} \Big) \]
The antecedent expresses that $\la s_1,\ldots,s_n,e \ra$ and $\la t_1,\ldots,t_n,f \ra$ are balanced.

Note that this logic includes Totality, which the other logics in this paper will not include. Finite Cancellation and Generalized Finite Cancellation are equivalent in the presence of Totality, but Generalized Finite Cancellation does not imply Finite Cancellation in the absence of Totality \cite{HTHI16}. \FinLogic has long been known to be sound and complete for finite cardinality models. The key result for proving completeness is the following representation theorem, which we will also make use of later in the more general cases.

\begin{theorem}[Kraft, Pratt, Seidenberg \cite{Kraft1959}, Theorem 2]
		\label{thm:KPS-representation}
		For any finite Boolean algebra $B$ with $\top$ as the top element and $\bot$
		the bottom element and any binary relation $\cge$ on $B$, there is a
		probability measure $\mu$ on $B$ such that for all $a, b \in B$, $a \cge b$ iff $\mu(a) \ge \mu(b)$, if and only if the following conditions are satisfied:
		\begin{itemize}
			\item not $\bot \cge \top$;
			\item for all $b \in B$, $b \cge \bot$;
			\item $\cge$ is a total pre-order, i.e., $\cge$ is transitive, and for any $a, b \in B$, $a \cge b$ or $b \cge a$;
			\item for any two sequences of elements $a_1, a_2, \dots , a_n$ and $b_1, b_2,
			\dots , b_n$ from $B$ of equal length, if every atom of $B$ is below (in the
			order of the Boolean algebra) exactly as many $a$'s as $b$'s, and if $a_i
			\cge b_i$ for all $i \in\{1, \dots , n-1\}$, then $b_n \cge a_n$.
		\end{itemize}
	\end{theorem}
 
    The final condition is the expression (in a Boolean algebra) of FC. We express that the sequences $\la a_1,\ldots,a_n \ra$ and $\la b_1,\ldots,b_n \ra$ are balanced similarly to the way we expressed it in the logic with $\mc{S}_i$ and $\mc{T}_i$.
    
    \begin{definition}
        Given elements $a_1,\ldots,a_n$ and $b_1,\ldots,b_n$ of a Boolean algebra $B$, we say that $\la a_1,\ldots,a_n \ra$ and $\la b_1, b_2,\dots , b_n \ra$ are \textit{atomically balanced} and write $\la a_1, a_2, \dots , a_n \ra =_0 \la b_1, b_2,\dots , b_n \ra$ if every atom of $B$ is below exactly as many $a$'s as $b$'s. If some $a$'s or $b$'s are repeated, we could multiplicity.
    \end{definition}

    This is independent of $B$, i.e., if $B'$ is some other Boolean algebra containing all of the elements, two sequences are atomically balanced in $B$ if and only if they are atomically balanced in $B'$. This includes, e.g., if $B'$ is an ideal of $B$ (i.e., a subalgebra except that it has a different top elements) containing all of the $a$'s and $b$'s.

\subsection{Logic for Dedekind-Finite Sets}

We now give the logic for reasoning about the cardinality of Dedekind-finite sets. This logic is the same as the logic \textsf{IP} of imprecise probability in \cite{HTHI} and \cite{AlonHeifetz}.\footnote{There is one small difference in that \textsf{IP} allows nesting, while in our language we do not. More precisely, our logic is the fragment of \textsf{IP} without nesting.}

\begin{definition}\label{dedfinlogic} The logic for Dedekind-finite sets \DedFinLogic is the logic for $\mathcal{L}$ with the following axiom schemas:
	\begin{enumerate}[label=(D\arabic*)]
		\item All substitution instances of classical propositional tautologies;\label{axiom:substitutions}
		\item $\neg \abs{\es}\geq\abs{\es^c}$ (Non-triviality);\label{axiom:non-triviality}
		\item $\abs{s}\geq\abs{\es}$ (Positivity);\label{axiom:positivity}
		\item $\abs{s}\geq\abs{s}$ (Reflexivity);\label{axiom:reflexivity}
		\item $\abs{s}\geq\abs{t} \land \abs{t}\geq\abs{u} \limplies \abs{s}\geq\abs{u}$ (Transitivity);\footnote{This follows from the next axiom, \ref{axiom:GFC}.}\label{axiom:transitivity}
		\item[(GFC)] GFC$_{k,l}(s_1,\ldots,s_k,e; t_1,\ldots,t_k,f)$;
		\makeatletter\protected@edef\@currentlabel{(GFC)}\makeatother\label{axiom:GFC} 
	\end{enumerate}
	and the following rules
	\begin{enumerate}
		\item[(R1)] If $\varphi$ and $\varphi \to \psi$ are theorems, then so is $\psi$. 
        (Modus Ponens);\label{axiom:modus-ponens}
		\item[(R2)] If $t=0$ is provable in the equational theory of Boolean algebras, then $\abs{\es} \geq \abs{t}$ is a theorem. \label{axiom:emptyset}
	\end{enumerate}
\end{definition}

GFC$_{k,l}(s_1,\ldots,s_k,e; t_1,\ldots,t_k,f)$ is the formal expression of the GFC principle in much the same way that we expressed FC. The only difference is that we must include $l$ copies of $e$ and $f$. For each $j$ such that $1 \leq j \leq n$, define the term $\mc{S}_j$ as the union of the terms of the form $s_1^{c_{s,1}} \cap \ldots \cap s_k^{c_{s,k}} \cap e^{c_{e,1}} \cap \ldots \cap e^{c_{e,l}}$, where exactly $j$ many $c_i$'s are $c$ and the rest are empty. Similarly, define $\mc{T}_j$ by replacing the $s$'s with $t$'s and the $e$'s with $f$'s. Intuitively $\mc{S}_j$ denotes the set of elements which are in exactly $j$ many sets among $s_1,\ldots,s_k,\underbrace{e,\ldots,e}_{l}$. Then, GFC$_{k,l}(s_1,\ldots,s_k,e; t_1,\ldots,t_k,f)$ is defined by:

\[ \Big( \bigwedge\limits_{i=1}^k \mc{S}_i = \mc{T}_i \Big) \limplies \Big(\Big(\bigwedge\limits_{i=1}^k \abs{s_i} \geq \abs{t_i} \Big) \limplies \abs{e} \leq \abs{f} \Big) \] 

Soundness is almost immediate.

\begin{theorem}[Soundness]\label{thm:soundness}
	\DedFinLogic is sound with respect to Dedekind-finite cardinality models. 
\end{theorem}
\begin{proof}
	Most of the axioms and rules for \DedFinLogic, like non-triviality \ref{axiom:non-triviality}, reflexivity \ref{axiom:reflexivity}, and transitivity \ref{axiom:transitivity}, are clearly valid according to the semantics. So, in proving the soundness of \DedFinLogic, our main task is to show that axiom schema \ref{axiom:GFC} is valid. This is the content of Theorem $\ref{thm:GFC}$.
\end{proof}

\subsection{Logic for Arbitrary Sets}

For reasoning about any sets, we simply replace the GCF principle by CGFC.

\begin{definition}\label{cardcomplogic} The logic for \CardCompLogic is the logic for $\mathcal{L}$ the same axioms schemas and rules as \DedFinLogic, except that the axiom GFC is replaced by CGFC:
	\begin{enumerate}[label=(D\arabic*)]
		\item[(CGFC)] CGFC$_{k,l,T}(s_1,\ldots,s_k,e; t_1,\ldots,t_k,f;u_\sigma : \sigma\in T)$\makeatletter\protected@edef\@currentlabel{(CGFC)}\makeatother\label{axiom:CGFC}.
	\end{enumerate}
	where $T$ is a finite full binary tree.
	In the notation of the previous subsection, the principle CGFC$_{k,l,T}(s_1,\ldots,s_k,e; t_1,\ldots,t_k,f;u_\sigma : \sigma \in T)$ is defined by:
	\[ \left[\Big( \bigwedge\limits_{i=1}^k \mc{S}_i = \mc{T}_i \Big) \wedge \Big(\bigwedge\limits_{i=1}^k \abs{s_i} \geq \abs{t_i} \Big) \wedge \Big( \bigwedge_{i=1}^k \abs{s_i} \leq \abs{u_\varnothing}\Big) \wedge \theta(f;u_\sigma: \sigma \in T) \Big)\right] \limplies \abs{e} \leq \abs{f}.\]
	Here, $\theta(f;u_\sigma: \sigma \in T)$ expresses that element corresponding to the root node of the tree, $u_\varnothing$, is bounded above in cardinality by some finite multiple of $f$, as witnessed by the $u_\sigma$. That is, $\theta(f;u_\sigma: \sigma \in T)$ is the conjunction, over non-leaves $\sigma$ of $T$, of $\abs{u_\sigma} \leq \abs{u_{\sigma 0} \cup u_{\sigma 1}}$, and, for leaves $\sigma$ of $T$, of $\abs{u_\sigma} \leq \abs{f}$. Another way to think of this is that the $u_\sigma$ witness that, in the Boolean algebra generated by all of these elements, each $s_i$ is in the smallest Boolean ideal containing $f$ and closed downwards in cardinality.
	
\end{definition}

\begin{theorem}[Soundness]\label{thm:soundness-inf}
	\CardCompLogic is sound with respect to cardinality models. 
\end{theorem}
\begin{proof}
Our main task is to show that axiom schema \ref{axiom:CGFC} is valid. This is the content of Theorem $\ref{thm:CGFC}$.
\end{proof}


\section{Probability measures models for imprecise probability}\label{sec:prop-measures-models}

In this section, we introduce models for imprecise probability. These are models for reasoning about comparisons between the probabilities of events in a situation where there is uncertainty about the true probabilities. In a model, we represent this by having a set of probability measures, and say that one event is more likely than another (written $E \succsim F$) if it is more likely according to every probability measure.
\begin{definition}
	A \textit{probability measures model} is a triple $\innerprod{W,P,V}$ such that $W$ is a set of states, $P$ is a set of finitely-additive probability measures on $\pset{W}$, and $V : \Phi \to \pset{W}$, where $\Phi$ is as in Definition \ref{def:language}. We define, for any $E,F \in \pset{W}$, \[\ E \precsim F \iff \forall \mu \in P,\ \mu(E) \leq \mu(F).\]
\end{definition}
It was shown by Alon and Heifetz \cite{AlonHeifetz} (using a representation theorem of Alon and Lehrer \cite{AL14} and also independently Rios Insua \cite{Insua92}) that \DedFinLogic (which they thought of as a logic of probabilistic reasoning) is sound and complete with respect to such models. We will use this completeness result as a step in our completeness result, by showing that one can transform a probability measures model into a Dedekind-finite pure cardinality model.

However, we will want our measures to take whole number values. Given a probability measures model, we can adjust the measures slightly to take rational values, and then clear denominators to obtain a $\mathbb{N}_{\geq 0}$-valued measure. This process does not change any of the comparisons between events. Thus probability measures models and the finitary measures models defined as follows are essentially equivalent.

\begin{definition}
	A \textit{finitary measures model} is a triple $\innerprod{W,P,V}$ such that $W$ is a set of states, $P$ is a set of finitely-additive measures on $\pset{W}$ taking values in $\mathbb{N}_{\geq 0}$, and $V : \Phi \to \pset{W}$. We define, for any $E,F \in \pset{W}$, \[\ E \precsim F \iff \forall \mu \in P,\ \mu(E) \leq \mu(F).\]
\end{definition}

The semantics for these models are defined very similarly to the semantics in Definition \ref{defn-of-model} for our cardinality models.

\begin{definition}
	Given a finitary (or probability) measures model $\mathcal{N} = \innerprod{W,P,V}$, we define a function $\hat{V}$, which assigns to each set term a set in $\pset{W}$, by:
	\begin{itemize}
		\item $\hat{V}(a) = V(a)$ for $a\in \Phi$
		\item $\hat{V}(u^c) = W - \hat{V}(u)$
		\item $\hat{V}(u\cap v) = \hat{V}(u)\cap\hat{V}(v)$
	\end{itemize}
	We then define a satisfaction relation $\models$ with the usual clauses for propositional variables and Boolean connectives, plus: $\mc{N} \models \abs{v} \leq \abs{u}$ if and only if $(\hat{V}(v)) \precsim (\hat{V}(u))$.
\end{definition}

The usual language for probability comparison uses $E \precsim F$ instead of $|E| \leq |F|$, but this is purely a notational difference.
\begin{theorem}[Alon and Heifetz \cite{AlonHeifetz}, using Alon and Lehrer \cite{AL14} and Rios Insua \cite{Insua92}]\label{sound-and-complete:prob}
	\DedFinLogic is sound and complete with respect to probability measures models (and hence with respect to finitary measures models).
\end{theorem}

When dealing with Dedekind-infinite sets, we will make use of models where the measures can take value $\infty$. We put $\infty \leq \infty$, so that if $\mu(E) = \infty$ and $\mu(F) = \infty$, then $\mu(E) \leq \mu(F)$. These have no obvious interpretation in terms of imprecise probability and, as far as we know, have not appeared in the literature.

\begin{definition}
	An \textit{infinitary measures model} is a triple $\innerprod{W,P,V}$ such that $W$ is a set of states, $P$ is a set of finitely-additive measures on $\pset{W}$ taking values in $\mathbb{N}_{\geq 0} \cup \{\infty\}$, and $V : \Phi \to \pset{W}$. We define, for any $E,F \in \pset{W}$, \[\ E \precsim F \iff \forall \mu \in P,\ \mu(E) \leq \mu(F).\]
	We define the satisfaction relations as for finitary measures models.
\end{definition}

We will show that \CardCompLogic is sound and complete with respect to infinitary measures models. As part of the proof, we will make use of the representation theorem of Kraft, Pratt, and Seidenberg we quoted above as Theorem \ref{thm:KPS-representation}.

\begin{theorem}\label{sound-and-complete:inf}
	\CardCompLogic is sound and complete with respect to infinitary measures models.
\end{theorem}
\begin{proof}
    Soundness is easy to see. The only non-trivial to check is CGFC, but this is easy to check.\footnote{Given a measure $\mu \in P$, the covering condition in CGFC$_{k,l,T}(s_1,\ldots,s_k,e; t_1,\ldots,t_k,f;u_\sigma : \sigma \in T)$ essentially says that if $\mu(s_i) = \infty$ for any $i$, then $\mu(f) = \infty$ as well, in which case the conclusion of CGFC clearly holds; otherwise, if each $\mu(s_i)$ is finite, checking CGFC is essentially the same as checking GCF.} We now prove completeness. Given a finite consistent set of formulas, over a finite set $\Phi$ of set labels, we can extend this set of a formulas to a complete set of formulas. Thus we may assume that we have a Boolean algebra $B$ together with a partial pre-ordering $\cge$ on $B$ satisfying non-triviality (not $\bot \cge \top$), positivity (for all $b \in B$, $b \cge \bot$), and
    \begin{itemize}
        \item (CGFC) for any two sequences of elements $a_1, a_2, \dots , a_n,\underbrace{e,\ldots,e}_{m}$ and $b_1, b_2,
			\dots , b_n,\underbrace{f,\ldots,f}_{m}$ from $B$ of equal length, if
			\begin{enumerate}
			    \item the sequences are atomically balanced,
       \[ \la a_1, a_2, \dots , a_n,\underbrace{e,\ldots,e}_{m} \ra =_0 \la b_1, b_2,
			\dots , b_n,\underbrace{f,\ldots,f}_{m} \ra\]
   (recalling that this means that every atom of $B$ is below, in the
			order of the Boolean algebra, exactly as many $a$'s and $e$'s as $b$'s and $f$'s, and counting multiplicity),
			    \item $a_i
			\cge b_i$ for all $i \in\{1, \dots , n\}$, and
			    \item there is a finite full binary tree $T$ and elements $\{c_\sigma : \sigma \in T\}$ such that
			    \begin{enumerate}
			        \item for each $i$, $a_i \cle c_\varnothing$,
			        \item for each non-leaf $\sigma$, $c_\sigma \cle c_{\sigma 0} \vee c_{\sigma 1}$, and
			        \item for each leaf $\sigma$, $c_\sigma \cle f$,
			    \end{enumerate}
			\end{enumerate}
			then $f \cge e$.
    \end{itemize}
    (Note that one consequence of CGFC is that if $a \leq b$ in the Boolean algebra, then $a \cle b$.) We must show that there is a set of $\infty$-valued measures $\Phi$ such that $a \cge b$ if and only if for all $\mu \in \Phi$, $\mu(a) \geq \mu(b)$.
    
	Given $a \in B$, let $B_a$ be the smallest ideal of $B$ which contains $a$ and is closed downwards under $\cle$. Then $B_a$ is exactly the ideal consisting of elements $b \in B$ such that there is a finite full binary tree $T$ and elements $\{c_\sigma : \sigma \in T\}$ such that
	\begin{enumerate}
		\item $b \cle c_\varnothing$,
		\item for each non-leaf $\sigma$, $c_\sigma \cle c_{\sigma 0} \vee c_{\sigma 1}$, and
		\item for each leaf $\sigma$, $c_\sigma \cle a$,
	\end{enumerate}
	We can think of $B_a$ as a Boolean algebra in its own right (though it may not be a subalgebra of $B$ as it may not have the same top elements).
	
	Now let $b \in B_a$ be such that $a \not\cge b$. We will show that $\cge$ can be extended to a total ordering $\cge^*$ on $B_a$ satisfying the conditions of Theorem \ref{thm:KPS-representation}, and such that $a \not\cge^* b$.
	
	\begin{claim}
	    There is a total ordering $\cge^*$ on $B_a$ satisfying the conditions of Theorem \ref{thm:KPS-representation}, and such that $a \not\cge^* b$.
	\end{claim}
	\begin{proof}
	    First, let $\cge_0$ be the closure of $\cge$ under GFC. This remains a partial pre-order, and still satisfies reflexivity and positivity. We will argue that $a \not\cge_0 b$, which also implies non-triviality. If we had $a \cge_0 b$, this would be because there is, in $B_a$ with the ordering $\cge$, an instance of GFC 
     \[(x_1,\ldots,x_n,a,\ldots,a) =_0 (y_1,\ldots,y_n,b,\ldots,b)\]
	    with $y_i \cge x_i$. (Note that any chain of reasoning using multiple instances of GFC can be combined into a single instance; thus, if  $a \cge_0 b$, it follows from a single instance of GFC.) Because each $x_i$ is in $B_a$, this is actually an instance of CGFC, and so it would be the case that $a \cge b$ in $B_a$. As this is not the case, we conclude that $a \not\cge_0 b$.
	
	    Now, one by one, we consider each pair of elements $c,d$ and extend our current partial pre-order $\cge_s$ to a new partial pre-order $\cge_{s+1}$ such that either $c \cge_{s+1} d$ or $d \cge_{s+1} c$. We maintain throughout the fact that $\cge_{s+1}$ satisfies GFC and has $a \not\cge_s b$. Thus, after considering all pairs of elements $c,d$, we will obtain a total pre-order $\cge^*$ satisfying GFC and such that a $\not\cge^* b$.
	
	    Considering the pair $c,d$, to show that we can extend $\cge_s$ to $\cge_{s+1}$, we argue by contradiction. We can consider the two partial pre-orders $\cge'$ and $\cge''$ obtained by adding to $\cge_s$ that $c \cge' d$ and that $d \cge'' c$ respectively, and then closing under GFC. If we have both that $a \cge' b$ and that $a \cge'' b$, then these follow from instances of GFC. First, $a \cge' b$ follows from an instance of GFC consisting of atomically balanced sequences
	    \[\la \underbrace{c,\ldots,c}_{\ell_1},x_1,\ldots,x_{n_1},\underbrace{b\ldots,b}_{m_1} \ra =_0 \la \underbrace{d,\ldots,d}_{\ell_1},y_1,\ldots,y_{n_1},\underbrace{a,\ldots,a}_{m_1} \ra \]
	    with $x_i \cge y_i$. Second, $a \cge'' b$ follows from an instance of GFC consisting of atomically balanced sequences
	    \[\la \underbrace{d,\ldots,d}_{\ell_2},u_1,\ldots,u_{n_2},\underbrace{b\ldots,b}_{m_2} \ra =_0 \la \underbrace{c,\ldots,c}_{\ell_2},v_1,\ldots,v_{n_2},\underbrace{a,\ldots,a}_{m_2} \ra \]
	    such that $u_i \cge v_i$. But then we can take $\ell_2$ copies of the first instance of GFC and $\ell_1$ copies of the second instance and combine them into a single instance of GFC. Cancelling out $\ell_1 \ell_2$ $c$'s and $d$'s from each side, we get an instance of GFC of the form
     \begin{align*}
         \la \underbrace{x_1,\ldots,x_1}_{\ell_2},\ldots,&\underbrace{x_{n_1},\ldots,x_{n_1}}_{\ell_2},\underbrace{u_1,\ldots,u_1}_{\ell_1},\ldots,\underbrace{u_{n_2},\ldots,u_{n_2}}_{\ell_1},\underbrace{b\ldots,b}_{\ell_2 m_1 + \ell_1 m_2} \ra 
         \\     
         &=_0 \la \underbrace{y_1,\ldots,y_1}_{\ell_2},\ldots,\underbrace{y_{n_1},\ldots,y_{n_1}}_{\ell_2},\underbrace{v_1,\ldots,v_1}_{\ell_1},\ldots,\underbrace{v_{n_2},\ldots,v_{n_2}}_{\ell_1},\underbrace{a,\ldots,a}_{\ell_2 m_1+\ell_1 m_2} \ra.
     \end{align*}
     This is an instance of GFC using only $\cge_s$, from which we could conclude that $a \cge_s b$. But this is not the case, giving a contradiction. Thus the desired extension $\cge_{s+1}$ must exist.
	
	    By repeatedly applying this argument, we conclude that $\cge$ can be extended to a total ordering $\cge^*$ on $B_a$, satisfying GFC, such that $a \not\cge b$.
	\end{proof}
	
	Then by Theorem \ref{thm:KPS-representation} there is a measure $\mu$ on $B_a$ such that 
	\[x \cge^* y \; \Longleftrightarrow\; \mu(x) \ge \mu(y).\]
	In particular, $\mu(a) < \mu(b)$. We can extend $\mu$ to all of $B$ by setting it to be $\infty$ on $B - B_a$. Since this measure $\mu$ was chosen specifically for $a$ and $b$, call it $\mu_{a,b}$.
	
	Also, for each $a$, let $\rho_a$ be the measure on $B$ defined to be $0$ on $B_a$, and $\infty$ on $B - B_a$. Then let
	\[ \Phi = \{\mu_{a,b} : b \in B_a, \; a \not \cge b\} \cup \{\rho_a: a \in B\}.\]
	Then $x \cge y$ if and only if, for all $\mu \in \Phi$, $\mu(x) \geq \mu(y)$.
\end{proof}

\section{Urelements and Permutation Models}\label{sec:urelements}

Another intermediate step in our completeness proofs will be to consider set theory with urelements.\footnote{We use the term urelement rather than the more common atom because we will reserve atom to refer to elements of a Boolean algebra} We use $\ZFA$ to refer the Zermelo-Frenkel set theory with urelements. We will generally use the symbol $A$ to refer to the set of urelements. The axioms of $\ZFA$ are the same as the axioms of $\ZF$, modified in the natural way to accommodate urelements; for example, modifying extensionality to apply only to sets and not to urelements (see Chapter 4 of \cite{Jech73} for details).

Our use of urelements will be in the form of permutation models (see Definition \ref{def:permutation models}). The intuition for these is as follows. We begin with a model of $\ZFA + \AC$. This model contains certain sets which distinguish between the different atoms, for example, the set of all atoms can be split into two disjoint infinite subsets $A = A_1 \sqcup A_2$. Now we think of ourselves as being unable to distinguish between different atoms, and we throw away all sets, such as $A_1$ and $A_2$, which can only be defined by distinguishing between atoms. Thus we get a submodel of the original model which is a model of $\ZFA$ but usually not of $\AC$. Permutation models were originally introduced by Fraenkel to show, about 40 years before Cohen's proof of the independence of the Axiom of Choice for $\ZF$, that the Axiom of Choice was not provable in $\ZFA$ (see Chapter 4 of \cite{Jech73} for details).

In the definitions below we will need the cumulative hierarchy $\mc{P}^\alpha(S)$ above a given set $S$, defined inductively as follows
\begin{align*}
	\mc{P}^0(S) &= S,\\
	\mc{P}^{\alpha+1}(S) &= \mc{P}^{\alpha}(S) \cup \pset{\mc{P}^{\alpha}(S)},\\
	\mc{P}^\lambda(S) &= \bigcup_{\beta<\lambda} \mc{P}^\beta(S) \ \ \ \ \text{($\lambda$ limit)}
\end{align*}
A model of $\ZFA$ has a \textit{kernel}, namely that part of the model consisting only of pure sets: $\mc{P}^\infty(\varnothing)$.

\begin{definition}[Permutation Models, see section 4.2 of \cite{Jech73}]\label{def:permutation models}
	Consider set theory with urelements and let $A$ be the set of urelements. Let $\pi$ be a permutation of the set $A$. Using the hierarchy of $\mc{P}^\alpha(A)$'s, we can extend $\pi$ (either by $\in$-recursion or by recursion on the rank of $x$) to act on arbitrary sets $x$ as as follows:
	\[\pi(\es) = \es, \ \pi(x) = \pi[x] = \set{\pi(y) : y \in x}. \]
	Then $\pi$ becomes an \textit{$\in$-automorphism} of the universe and one can easily verify the following facts about $\pi$:
	\begin{itemize}
		\item[(a)] $x\in y \iff \pi x \in \pi y$;
		\item[(b)] $\phi(x_1,\ldots,x_n) \iff \phi(\pi x_1,\ldots,\pi x_n)$;
		\item[(c)] rank($x$) = rank($\pi x$);
		\item[(d)] $\pi\set{x,y} = \set{\pi x, \pi y}$ and $\pi(x,y) = (\pi x, \pi y)$;
		\item[(e)] If $R$ is a relation then $\pi R$ is a relation and $(x,y)\in R \iff (\pi x,\pi y) \in \pi R$;
		\item[(f)] If $f$ is a function on $X$ then $\pi f$ is a function on $\pi X$ and $(\pi f)(\pi x) = \pi(f(x))$;
		\item[(g)] $\pi x = x$ for every $x$ in the kernel;
		\item[(h)] $(\pi \cdot \rho)x = \pi(\rho(x))$.
	\end{itemize}
	
	
	Let $\mc{G}$ be a group of permutations of $A$. A set $\mc{F}$ of subgroups of $\mc{G}$ is a \textit{normal filter} on $\mc{G}$ if for all subgroups $H,K$ of $\mc{G}$:
	\begin{itemize}
		\item[(i)] $\mc{G} \in \mc{F}$;
		\item[(ii)] if $H \in \mc{F}$ and $H \subseteq K$, then $K \in \mc{F}$;
		\item[(iii)] if $H \in \mc{F}$ and $K \in \mc{F}$, then $H\cap K \in \mc{F}$;
		\item[(iv)] if $\pi \in \mc{G}$ and $H \in \mc{F}$, then $\pi H \pi^{-1} \in \mc{F}$;
		\item[(v)] for each $a\in A$, $\set{\pi \in \mc{G} : \pi a = a} \in \mc{F}$.
	\end{itemize}
    For each $x$, let $\sym_{\mc{G}}(x) = \set{\pi \in \mc{G} : \pi x = x}$; note that $\sym_{\mc{G}}(x)$ is a subgroup of $\mc{G}$.
	
	Let $\mc{G}$ and $\mc{F}$ be fixed. We say that $x$ is \textit{symmetric} if $\sym_{\mc{G}}(x) \in \mc{F}$. The class \[\mc{U} = \set{x : x \text{ is symmetric and } x\subseteq \mc{U}}\]
	consists of all \textit{hereditarily symmetric} objects. So far $\mc{U}$ is just a class; however, we can prove that it is in fact a model of $\ZF$ (see Theorem 4.1 of \cite{Jech73}). We call $\mc{U}$ a \textit{permutation model}.
	
	For our purposes, it suffices to consider the following simple type of permutation models: Let $\mc{G}$ be a group of permutations of $A$. A family $I$ of subsets of $A$ is a \textit{normal ideal} if for all subsets $E,F$ of $A$:
	\begin{itemize}
		\item[(i)] $\es \in I$;
		\item[(ii)] if $E\in I$ and $F \subset E$, then $F \in I$;
		\item[(iii)] if $E \in I$ and $F \in I$, then $E \cup F \in I$;
		\item[(iv)] if $\pi \in \mc{G}$ and $E \in I$, then $\pi[E] \in I$;
		\item[(v)] for each $a \in A$, $\set{a} \in I$.
	\end{itemize}
    For each $x$, let $\fix_{\mc{G}}(x) = \set{\pi \in \mc{G} : \pi y = y \text{ for all } y \in x}$; note that $\fix_{\mc{G}}(x)$ is a subgroup of $\mc{G}$.
	
	Given a normal ideal $I$, let $\mc{F}$ be the filter on $\mc{G}$ generated by the subgroups of $\fix_{\mc{G}}(E), E\in I$. $\mc{F}$ is a normal filer, and so it defines a permutation model $\mc{U}$. Note that $x$ is symmetric if and only if there exists $S \in I$ such that \[\fix_{\mc{G}}(S) \subseteq \sym_{\mc{G}}(x).\] We say that $S$ is the \textit{support} of $x$.
\end{definition}

We now introduce urelement cardinality models, which function as the $\ZFA$ analogue of our pure cardinality models. In our completeness proof, we will use the technology of permutation models to build urelement cardinality models, which we will then transform into pure cardinality models using the Jech-Sochor Embedding Theorem (Theorem \ref{thm:first-embedding-thm}).

\begin{definition}
	An \textit{urelement cardinality model} is a quadruple $\mathcal{M} = \innerprod{\mathcal{U},X,\mathcal{F},V}$, where $\mathcal{U}$ is a model of $\ZFA$, $X$ is a non-empty set in $\mathcal{U}$, $\innerprod{X, \mathcal{F}}$ is a field of sets in $\mc{U}$, and $V : \Phi \to \mathcal{F}$. We say that an urelement cardinality model is a \textit{permutation urelement cardinality model} if the underlying model $\mc{U}$ of $\ZFA$ is a permutation model.\footnote{\label{footnote-ur}We could also ask for $X$ to be a non-empty set of urelements, but this adds unnecessary though ultimately trivial complications to some of our proofs, so we do not ask for this.}
\end{definition}

\begin{theorem}[Jech-Sochor Embedding Theorem; see Theorem 6.1 of \cite{Jech73}]\label{thm:first-embedding-thm}
	Let $\mc{Y}$ be a model of $\ZFA+\AC$, let $A$ be the set of all atoms of $\mc{Y}$, let $\mc{K}$ be the kernel of $\mc{Y}$ and let $\alpha$ be an ordinal in $\mc{Y}$. For every permutation model $\mc{U} \subseteq \mc{Y}$ (a model of $\ZFA$), there exists a symmetric extension $\mc{U}^* \supseteq \mc{K}$ (a model of $\ZF$) and a set $\tilde{A} \in \mc{U}^*$ such that
	\[(\mc{P}^{\alpha}(A))^{\mc{U}} \text{ is $\in$-isomorphic to } (\mc{P}^{\alpha}(\tilde{A}))^{\mc{U^*}}.\]
\end{theorem}

\begin{theorem}\label{thm:urelement-to-pure}
        	For each permutation urelement cardinality model $\mathcal{M}$, there is a symmetric pure cardinality model $\mathcal{N}$ such that $\mathcal{M} \equiv_{\mc{L}} \mathcal{N}$. Moreover, if $\mathcal{M}$ is a Dedekind-finite permutation urelement  cardinality model, the we can choose $\mathcal{N}$ to be a Dedekind-finite pure cardinality model.
\end{theorem}

Recall that $\mathcal{M} \equiv_{\mc{L}} \mathcal{N}$ means that $\mc{M}$ and $\mc{N}$ satisfy the same formulas. Thus we are showing that each permutation urelement cardinality model is equivalent to a symmetric pure cardinality model.

\begin{proof}
	Let $\mc{M} = (\mc{U},X,\mc{F},V)$ be a permutation urelement cardinality model. Thus $\mc{U}$ is a permutation model of $\ZFA$. We will construct a model $\mc{U}^*$ of $\ZF$ without atoms, a set $X^* \in \mc{U}^*$, and a field of sets $\mc{F}^*$ on $X^*$ such that there is a bijection $x \mapsto x^*$ from $X \to X^*$, inducing an isomorphism of Boolean algebras $Y \mapsto Y^*$ from $\mc{F} \to \mc{F}^*$. Moreover, this map will respect cardinality: we will have that $|E| \geq |F|$ in $\mc{U}$ if and only if $|E^*| \geq |F^*|$ in $\mc{U}^*$. Then, defining $V^*(p) = V(p)^*$, we will get that $\mc{N} = (\mc{U}^*,X^*,\mc{F}^*,V^*)$ satisfies $\mc{M} \equiv_{\mc{L}} \mc{N}$.
	
	Let $A$ be the set of all urelements of $\mc{U}$ and let $\mc{K}$ be the kernel of $\mc{U}$. Let $\beta$ be an ordinal in $\mc{U}$ which is sufficiently large that $X,\mc{F},\omega \in (\mc{P}^{\beta}(A))^{\mc{U}}$. Let $\alpha = \beta + 3$; note that for any $E,F \subseteq X$, and any injection $f: E \to F$, viewing $f$ as a set of ordered pairs we have that $f \in (\mc{P}^{\alpha}(A))^{\mc{U}}$. Then by The Jech-Sochor Embedding Theorem (Theorem \ref{thm:first-embedding-thm}), there is a model of $\ZF$, $\mc{U}^* \supseteq \mc{K}$, and a set $\tilde{A}$ in $\mc{U^*}$ such that \[(\mc{P}^{\alpha}(A))^{\mc{U}} \text{ is $\in$-isomorphic to } (\mc{P}^{\alpha}(\tilde{A}))^{\mc{U^*}}\]
	Let $\phi$ be this map. A vital property of $\phi$ which we will use implicitly throughout the rest of the proof is that, for any set $S$ in $(\mc{P}^{\alpha}(A))^{\mc{U}}$,
	\[ \phi(S) = \{\phi(x) : x \in S\}.\]
	That is, $\phi(S)$ cannot get any new elements which are not in the range of $\phi$. This relies not just on the fact that $\phi$ is an $\in$-isomorphism, but on the particular range of $\phi$. (Note that this fact does not apply to atoms $a$ of $\mc{U}$; in this case, $\phi(a)$ is a non-empty set.) Another property of $\phi$ is that it acts level-by-level: for each $\gamma \leq \alpha$, $\phi$ maps
	\[(\mc{P}^{\gamma}(A))^{\mc{U}} \longrightarrow (\mc{P}^{\gamma}(\tilde{A}))^{\mc{U^*}}.\]
	
	For each $x \in X$, define $x^*$ in $\mc{U}^*$ by $x^* = \phi(x)$. We let
	\[ X^* = \phi(X) = \{x^* : x \in X\}\]
	and for any $Y \subseteq X$, we have $Y^* \subseteq X^*$ defined by
	\[ Y^* = \phi(Y) = \{x^* : x \in Y\}.\]
	We also define
	\[\mc{F}^* = \phi(\mc{F}) = \set{Y^* : Y\in \mc{F}}.\]
	Note that $Y \mapsto Y^*$ is an isomorphism of Boolean algebras between the fields of sets $\mc{F}$ and $\mc{F}^*$.
	
	Now, suppose that $\abs{E}\geq \abs{F}$ in $\mc{U}$; that is, suppose that there is an injection $f : F \to E$. Then $\phi(f)$ will be an injection from $F^* = \phi(F)$ to $E^* = \phi(E)$ as
	\[ f(a) = b \Longleftrightarrow \phi(f)(\phi(a))=\phi(b).\]
	(Note that $f$ is in the domain of $\phi$ by choice of $\alpha$.) Thus $\abs{E^*}\geq \abs{F^*}$.
	
	Conversely, suppose that $\abs{E^*}\geq \abs{F^*}$; that is, suppose that there is an injection $g^* : F^* \to E^*$. First note that $g^*$ is in the range of $\phi$ because $E^*$ and $F^*$ are in $(\mc{P}^{\beta}(\tilde{A}))^{\mc{U^*}}$ (we use here the fact that $\phi$ maps level-by-level), and so $g^*$ is in $(\mc{P}^{\alpha}(\tilde{A}))^{\mc{U^*}}$. Then let $g$ be the preimage of $g^*$ under $\phi$, $g^* = \phi(g)$. Then, as $\phi$ is an $\in$-isomorphism, $g$ will be an injection from $E$ to $F$ and so $\abs{E} \geq \abs{F}$. 
	
	We now check the moreover clause. Suppose that $Y^*$ is Dedekind-infinite. Let $\mathbf{b} : \omega \to Y^*$ be an $\omega$-sequence in $Y^*$. Then $\mathbf{b} \in \mathcal{P}^\alpha(\tilde{A})^{\mathcal{U}^*}$. Since $\mathbf{b}$ is in the image of the $\in$-isomorphism  $\phi$, there is a corresponding $\omega$-sequence $\mathbf{a} : \omega \to Y$ such that $\phi(\mathbf{a}) = \mathbf{b}$. Thus $Y$ is Dedekind-infinite. By the contrapositive, we have that if $Y$ is Dedekind-finite, then $Y^*$ is Dedekind-finite.
\end{proof}
Note that each pure cardinality model is already an urelement cardinality model, and so we get:
\begin{corollary}\label{cor:urelement-same-as-pure}
	The logic of urelement cardinality models is the same as the logic of pure cardinality models. The same is true for Dedekind-finite urelement cardinality models and Dedekind-finite pure cardinality models.
\end{corollary}

\section{Representation Theorems}\label{sec:prob-to-set}

	\begin{theorem}\label{thm:prob-to-urelement}
	For each finite(-size) infinitary measures model $\mc{M}$, there is an urelement cardinality model $\mc{N}$ such that $\mathcal{M} \equiv_{\mc{L}} \mathcal{N}$. Moreover, if $\mc{M}$ is in fact a finitary measures model, then we can take $\mathcal{N}$ to be a Dedekind-finite urelement cardinality model.
\end{theorem}
\begin{proof}
	Let $\mc{M} = \innerprod{W,P,V}$ be an infinitary  measures model with $W$ and $P$ finite. We will construct an urelement cardinality model $\mathcal{N} = (\mc{U},X,\mc{F},V^*)$ such that $\mathcal{M} \equiv_{\mc{L}} \mathcal{N}$.
	
	In fact, we will define $\mc{U}$, $X$, and $\mc{F}$ and show that there is a map $E \mapsto E^*$ from $\pset{W} \to \mc{F}$ which is an isomorphism of Boolean algebras such that for $E,F \subseteq W$,
	\[ E \precsim F \Longleftrightarrow \mc{U} \models |E^*| \leq |F^*|.\]
	Given $s \in \Phi$, and $V(s)  = E \subseteq W$, we define $V^*(s) = E^* \in \mc{F}$. The following two claims are then easy to prove by induction on terms and formulas, and they imply that $\mc{M} \equiv_{\mc{L}} \mc{N}$.
	
	\begin{claim}
		If $t$ is a term and $\hat{V}(t) = E$, then $\hat{V}^*(t) = E^*$.
	\end{claim}
	
	\begin{claim} $\mc{M} \models \varphi \Longleftrightarrow \mc{N} \models \varphi.$
	\end{claim}
	
	To complete the proof, we will now define $\mc{U}$, $X$, and $\mc{F}$ as above, beginning with $\mc{U}$. Let $\mc{Y}$ be a model of $\ZFA+\AC$ and let $A$ be the set of urelements in $\mc{Y}$. The model $\mc{U}$ will be a permutation model (a model of $\ZFA$) defined from $\mc{Y}$. Much of our construction will follow the general argument of Theorem 11.1 of \cite{Jech73}, where given a partial order $(P,\preceq)$ in the kernel of $\mc{Y}$, Jech constructs a permutation model $\mc{U}$ containing sets $(S_p)_{p \in P}$ such that $p \preceq q$ if and only if $|S_p| \leq |S_q|$ in $\mc{U}$.
	
	Consider the poset $(P,\preceq)$, where $P$ is the set of measures, with $\preceq$ being the ordering that no two of them are comparable. Since $(P,\preceq)$ is finite, a copy of it is contained in the kernel of $\mc{Y}$.
	
	We then follow the construction of Theorem 11.1 of \cite{Jech73} using ($P$,$\preceq$) as our partial order. Assume that the set $A$ of urelements has cardinality $\abs{A} = \abs{P}\cdot\aleph_0$ and let $\set{a_{\mu,n} : \mu\in P, \; n\in\omega}$ be an enumeration of $A$. For each $\mu\in P$, let
	\[ A_{\mu} = \{ a_{\mu,n} : n \in \omega\}.\]
	
	Let $\mathcal{G}$ be the group of all permutations $\pi$ of $A$ such that $\pi(A_{\mu}) = A_{\mu}$ for each $\mu\in P$; that is, if $\pi a_{\mu,n} = \pi a_{\rho,m}$ then $\mu=\rho$. Let $\mc{F}$ be the filter on $\mc{G}$ given by the ideal of finite subsets of $A$. Since $\mc{F}$ is a normal filter on $\mc{G}$, it defines a permutation model $\mc{U}$ consisting of all hereditarily symmetric elements of $\mc{Y}$. Furthermore, each $x\in \mc{U}$ has a finite support $S\subseteq A$ such that $\sym(x) \supseteq \fix(S)$.
	
	\begin{claim}\label{claim:A_mu_Dedfin}
		$A_\mu$ is amorphous, hence Dedekind-finite.
	\end{claim}
	\begin{proof}
		Suppose $A_\mu$ is not amorphous. Then it can be split up into two disjoint infinite sets $B_\mu$ and $C_\mu$. If $B_\mu$ is preserved in $\mc{U}$, then it must have some finite support $S$. Furthermore, we know that any permutation $\pi \in \fix(S)$ would fix $B_\mu$. We choose $b \in B_\mu$ and $c \in C_\mu$ such that $b,c \notin S$. We then let $\pi$ be the permutation that swaps $b$ and $c$ and fixes everything else; that is, $\pi(b) = c$, $\pi(c) = b$, and $\pi(a)=a$ for all $a\neq b,c$. Note that $\pi$ fixes each element of $S$, and so $\pi$ fixes $B_\mu$ as a set. However, $\pi(b) = c \notin B_\mu$, which is a contradiction. 
		Thus $A_\mu$ must be amorphous.
	\end{proof}
	
	As a result of this claim, if $f$ is a map whose domain (or range) is a subset of $A_\mu$, then its domain (or range) must be finite or cofinite in $A_\mu$.
	
	\begin{claim} There is a partial injection $f \colon D \subseteq A_\mu \to A_\rho$ defined on an infinite subset $D$ of $A_\mu$ if and only if $\mu = \rho$.
	\end{claim}
	\begin{proof}
		For the forward direction, suppose $\mu \neq \rho$ and that there is a partial injection $\varphi : \dom(\varphi) \subseteq A_\mu \to A_\rho$ in $\mc{U}$. Since $\varphi$ is in $\mc{U}$, it must be finitely supported. Let $S$ be a finite support of $\varphi$. We choose $a,a' \in \dom(\varphi) \subseteq A_\mu$ such that $a,a' \notin S$. 
		Let $b = \varphi(a)$ and $b' = \varphi(a')$. We then let $\pi$ be the permutation that swaps $a$ and $a'$ and fixes everything else; that is, $\pi(a) = a'$, $\pi(a')=a$ and $\pi(c) = c$ for all $c \neq a,a'$. Since $\pi$ fixes $\varphi$, we have that
		\[ \pi(b) = \pi( \varphi(a) ) = \varphi(\pi a) = \varphi(a') = b' \]
		Since $b \neq a,a'$ (as $A\mu$ and $A_\rho$ are disjoint), we have that $\pi b = b$. From the above equation it also follows that $\pi b = b'$. So, we conclude that $b = b'$. However, this would mean that $\varphi(a) = \varphi(a')$, contradicting the injectivity of $\varphi$.
		
		For the converse direction, if $\mu = \rho$, then the identity on $A_\mu$ is an injection with empty support, i.e., an injection in $\mc{U}$.
	\end{proof}
	
	For each atomic event $E$ (i.e., $E$ is a singleton $\{w\}$), we can associate a set $f_\mu(E)$ of size $\mu(E)$ such that for any two distinct $E$ and $F$, the sets $f_\mu(E)$ and $f_\mu(F)$ are disjoint.\footnote{We can pick any sets for the $f_\mu(E)$, and they do not have to be in the kernel nor must they consist entirely of urelements. See footnote \ref{footnote-ur}; if we had made the alternate definition described there, then we would be more constrained at this point in the proof.} Since any other event is a union of atomic events, once $f_\mu(E)$ is chosen for each atom $E$, we can extend it to all sets by setting
	\[f_\mu(E) = \bigcup_{\text{$E \supseteq F$ an atom}} f_\mu(F).\]
	Then the map $E \mapsto f_\mu(E)$ is a Boolean algebra isomorphism and $|f_\mu(E)| = \mu(E)$.

	We then define $E^*$ to be
	\[ E^* = \bigcup_{\mu \in P} f_\mu(E) \times A_{\mu}.\]
	It follows that $E \mapsto E^*$ is an isomorphism of Boolean algebras.
	
	\begin{claim}
	    $E \precsim F$ if and only if $\abs{E^*} \leq \abs{F^*}$.
	\end{claim}
	\begin{proof}
	Suppose that $E \precsim F$, i.e. for any $\mu$ in $P$, $\mu(E)\leq \mu(F)$. It follows that $\abs{f_\mu(E)} = \mu(E) \leq \mu(F) = \abs{f_\mu(F)}$. Consequently,
	\[\abs{E^*} = \abs{\bigcup_{\mu \in P} f_\mu(E) \times A_{\mu}} \leq \abs{\bigcup_{\mu \in P} f_\mu(F) \times A_{\mu}} = \abs{F^*}. \]
	
	Now we suppose that $\mc{U} \models \abs{E^*} \leq \abs{F^*}$ as witnessed by an injection $g : E^* \to F^*$ in $\mc{U}$. For each $\mu \in P$, we argue that $\mu(E) \leq \mu(F)$, and thus $E \precsim F$. If $\mu(F)$ is infinite, then this is immediate. Otherwise, suppose that $\mu(F)$ is finite; we will define as follows an injection $h_\mu : f_\mu(E) \to f_\mu(F)$ witnessing that $\mu(E) \leq \mu(F)$.
	
	Given $x \in f_\mu(E)$, there is a copy $\set{x}\times A_\mu$ in $E^*$. We first argue that $g$ maps all but finitely much of $\set{x}\times A_\mu$ to some $\set{y}\times A_\mu$ in $F^*$. Consider, for each $y$ and $\rho$, the inverse image in $\set{x} \times A_\mu$ of $\{y\} \times A_\rho$. If $\rho \neq \mu$, then this inverse image is finite by the previous claim. Since there are only finitely many $y$'s and $\rho$'s, at least one of these inverse images, with $\rho = \mu$, must be infinite. Say that it is the inverse image of $\set{y}\times A_\mu$. Since $A_\mu$ is amorphous, this inverse image is in fact cofinite, and so all but finitely much of $\set{x}\times A_\mu$ maps to $\set{y}\times A_\mu$. Thus this $y$ is unique, and we define $h_\mu(x)$ to be $y$.
	
	
	
	
	We complete this argument by showing that $h_\mu$ must, in fact, be an injection. So, for a contradiction, suppose that there are $x, x'$ in $f_\mu(E)$ such that infinitely much of both $\set{x}\times A_\mu$ and $\set{x'}\times A_\mu$ maps into $\set{y}\times A_\mu$ in $F^*$. However, this would mean that the restricted images $g[\set{x}\times A_\mu]$ and $g[\set{x'}\times A_\mu]$ are both disjoint infinite subsets of $\set{y}\times A_\mu$, contradicting that $A_\mu$ is amorphous. This must mean that different $x$'s must map to different $y$'s, and so $h_\mu$ is an injection.
	\end{proof}
	
	Now that we have defined $\mc{U}$, we will define $X$ and $\mc{F}$ \[X := \bigcup_{E \subseteq W} E^*.\]
	Finally, we define $\mc{F} = \set{E^* \ | \ E \subseteq W}$ as the image of $\pset{W}$ under $*$. Then $\innerprod{X,\mc{F}}$ will be the field of sets in our urelement cardinality model.
	
	The model $\mathcal{N} = (\mc{U},X,\mc{F},V^*)$ as constructed above is an urelement cardinality model such that $\mathcal{M} \equiv_{\mc{L}} \mathcal{N}$, proving the main clause of the theorem. Finally, we check the moreover clause.
	
	\begin{claim}\label{claim:fin-implies-dedfin}
		If $\mc{M}$ is in fact a finitary measures model, then each element of $\mc{F}$ is Dedekind-finite.
	\end{claim}	
	\begin{proof} Recall that $E^* := \bigcup_{\mu \in P} f_\mu(E) \times A_{\mu}$. Since each $A_\mu$ is Dedekind-finite (Claim \ref{claim:A_mu_Dedfin}), and $\mu(E) = |f_\mu(E)|$ is finite, we have that $f_\mu(E) \times A_\mu$ is Dedekind-finite. Furthermore, since $P$ is finite, $E^*$ is a finite union of Dedekind-finite sets, and so each $E^*$ is Dedekind-finite; that is, each element of $\mc{F}$ is Dedekind-finite.
	\end{proof}
	\renewcommand{\qed}{}
\end{proof}

We now have all of the ingredients for our completeness proofs. 

{
\renewcommand{\thetheorem}{\ref{thm:completeness}}
\begin{theorem}[Completeness]
	\DedFinLogic (\ref{dedfinlogic}) is complete with respect to Dedekind-finite urelement cardinality models and also with respect to Dedekind-finite pure cardinality models. 
\end{theorem}
\addtocounter{theorem}{-1}
}
\begin{proof}
    By Theorem \ref{sound-and-complete:prob}, \DedFinLogic is complete with respect to finitary measures models. By  Theorem \ref{thm:prob-to-urelement}, any finite-size finitary measures model can be transformed into a Dedekind-finite urelement cardinality model that satisfies the same formulas. Thus \DedFinLogic is complete with respect to Dedekind-finite urelement cardinality models.
	
	Furthermore, recall that Corollary \ref{cor:urelement-same-as-pure} says that the logic of Dedekind-finite urelement cardinality models is the same as that of Dedekind-finite pure cardinality models. So, the completeness of \DedFinLogic with respect to Dedekind-finite pure cardinality models follows from the completeness of \DedFinLogic with respect to Dedekind-finite urelement cardinality models.
\end{proof}

{
\renewcommand{\thetheorem}{\ref{thm:completeness-inf}}
\begin{theorem}[Completeness]
	\CardCompLogic (\ref{cardcomplogic}) is complete with respect to urelement cardinality models and also with respect to pure cardinality models.
\end{theorem}
\addtocounter{theorem}{-1}
}
\begin{proof}
    By Theorem \ref{sound-and-complete:inf}, \CardCompLogic is complete with respect to infinitary measures models. By Theorem \ref{thm:prob-to-urelement}, any finite-size infinitary measures model can be transformed into an urelement cardinality model that satisfies the same formulas. Thus \CardCompLogic is complete with respect to urelement cardinality models.
	
	Furthermore, recall that Corollary \ref{cor:urelement-same-as-pure} says that the logic of urelement cardinality models is the same as that of pure cardinality models. So, the completeness of \CardCompLogic with respect to pure cardinality models follows from the completeness of \CardCompLogic with respect to urelement cardinality models.
\end{proof}

\bibliography{References}

\begin{thebibliography}{DHTH20}

\bibitem[AH14]{AlonHeifetz}
Shiri Alon and Aviad Heifetz.
\newblock The logic of {K}nightian games.
\newblock {\em Econ. Theory Bull.}, 2(2):161--182, 2014.

\bibitem[AL14]{AL14}
Shiri Alon and Ehud Lehrer.
\newblock Subjective multi-prior probability: {A} representation of a partial
  likelihood relation.
\newblock {\em Journal of Economic Theory}, 151:476--492, 2014.

\bibitem[Ber05]{Bernstein05}
Felix Bernstein.
\newblock Untersuchungen aus der {M}engenlehre.
\newblock {\em Math. Ann.}, 61(1):117--155, 1905.

\bibitem[Bur10]{Burgess2010}
John~P. Burgess.
\newblock {Axiomatizing the Logic of Comparative Probability}.
\newblock {\em Notre Dame Journal of Formal Logic}, 51(1):119 -- 126, 2010.

\bibitem[DC94]{DC94}
Peter~G. Doyle and John~Horton Conway.
\newblock Division by three, 1994.
\newblock https://arxiv.org/abs/math/0605779.

\bibitem[dCH91]{Cerro1991}
Luis~Fari{\~{n}}as del Cerro and Andreas Herzig.
\newblock A modal analysis of possibility theory.
\newblock In Rudolf Kruse and Pierre Siegel, editors, {\em Symbolic and
  Quantitative Approaches to Uncertainty}, pages 58--62, Berlin, Heidelberg,
  1991. Springer Berlin Heidelberg.

\bibitem[Ded88]{Dedekind1888}
Richard Dedekind.
\newblock Was sind und was sollen die {Z}ahlen?
\newblock {\em Vieweg, Braunschweig}, 1888.

\bibitem[DHTH20]{DHTH20}
Yifeng Ding, Matthew Harrison-Trainor, and Wesley~H. Holliday.
\newblock The logic of comparative cardinality.
\newblock {\em The Journal of Symbolic Logic}, 85(3):972–1005, 2020.

\bibitem[DP15]{Dubois1988}
Didier Dubois and Henri Prade.
\newblock Possibility theory and its applications: Where do we stand?
\newblock {\em Mathware and Soft Computing Magazine}, 18, 01 2015.

\bibitem[G\"75]{Gardenfors1975}
Peter G\"{a}rdenfors.
\newblock Qualitative probability as an intensional logic.
\newblock {\em Journal of Philosophical Logic}, 4(2):171--185, 1975.

\bibitem[HHI16]{HTHI16}
Matthew Harrison{-}Trainor, Wesley~H. Holliday, and Thomas Icard, III.
\newblock A note on cancellation axioms for comparative probability.
\newblock {\em Theory and Decision}, 80(1):159--166, 2016.

\bibitem[HTHI17]{HTHI}
Matthew Harrison-Trainor, Wesley~H. Holliday, and Thomas Icard, III.
\newblock Preferential structures for comparative probabilistic reasoning.
\newblock {\em Proceedings of the AAAI Conference on Artificial Intelligence},
  31(1), Feb. 2017.

\bibitem[Ins92]{Insua92}
David~Rios Insua.
\newblock On the foundations of decision making under partial information.
\newblock In John Geweke, editor, {\em Decision Making Under Risk and
  Uncertainty: New Models and Empirical Findings}, pages 93--100. Springer
  Netherlands, Dordrecht, 1992.

\bibitem[Jec66]{Jech66}
Thomas~J. Jech.
\newblock On ordering of cardinalities.
\newblock {\em Bull. Acad. Polon. Sci. S\'{e}r. Sci. Math. Astronom. Phys.},
  14:293--296 (loose addendum), 1966.

\bibitem[Jec73]{Jech73}
Thomas~J. Jech.
\newblock {\em The axiom of choice}.
\newblock Studies in Logic and the Foundations of Mathematics, Vol. 75.
  North-Holland Publishing Co., Amsterdam-London; American Elsevier Publishing
  Co., Inc., New York, 1973.

\bibitem[KPS59]{Kraft1959}
Charles~H. Kraft, John~W. Pratt, and A.~Seidenberg.
\newblock {Intuitive Probability on Finite Sets}.
\newblock {\em The Annals of Mathematical Statistics}, 30(2):408 -- 419, 1959.

\bibitem[Sco64]{Scott64}
Dana Scott.
\newblock Measurement structures and linear inequalities.
\newblock {\em Journal of Mathematical Psychology}, 1(2):233--247, 1964.

\bibitem[Seg71]{Segerberg1971}
Krister Segerberg.
\newblock Qualitative probability in a modal setting.
\newblock In J.E. Fenstad, editor, {\em Proceedings of the Second Scandinavian
  Logic Symposium}, volume~63 of {\em Studies in Logic and the Foundations of
  Mathematics}, pages 341--352. Elsevier, 1971.

\bibitem[Sie22]{Sierpinski22}
Wac{\l}aw Sierpi\'{n}ski.
\newblock Sur l'{\'e}galit{\'e} 2m = 2n pour les nombres cardinaux.
\newblock {\em Fundamenta Mathematicae}, 3:1--6, 1922.

\bibitem[Sie47]{Sierpinski47}
Wac{\l}aw Sierpi\'{n}ski.
\newblock Sur l'implication {$(2m\leq 2n)\to(m\leq n)$} pour les nombres
  cardinaux.
\newblock {\em Fund. Math.}, 34:148--154, 1947.

\bibitem[Tak68]{Takahashi67}
Motoo Takahashi.
\newblock On incomparable cardinals.
\newblock {\em Comment. Math. Univ. St. Paul.}, 16:129--142, 1967/68.

\bibitem[Tar49]{Tarski49}
Alfred Tarski.
\newblock Cancellation laws in the arithmetic of cardinals.
\newblock {\em Fund. Math.}, 36:77--92, 1949.

\bibitem[Tar65]{Tarski65}
Alfred Tarski.
\newblock On the {E}xistence of {L}arge {S}ets of {D}edekind {C}ardinals
  ({A}bstract 65{T}-432).
\newblock {\em Notices of the American Mathematical Society}, 12:719, 1965.

\end{thebibliography}
\bibliographystyle{alpha}

\end{document}